\documentclass[a4paper,11pt]{amsart}

\usepackage{amssymb}
\usepackage[all]{xy}
\usepackage[english]{babel}
\usepackage[latin1]{inputenc}
\usepackage{amsfonts}
\usepackage{amsthm}
\usepackage{amsmath}
\usepackage{amsfonts}
\usepackage{latexsym}
\usepackage{amssymb}
\usepackage{mathrsfs}
\usepackage[usenames]{color}

\newcommand{\pint}[1]{\displaystyle \left \langle #1 \right\rangle}

\newtheorem{lem}{Lemma}[section]
\newtheorem{thm}[lem]{Theorem}
\newtheorem{cor}[lem]{Corollary}
\newtheorem{rem}[lem]{Remark}
\newtheorem{rems}[lem]{Remarks}
\newtheorem{exa}[lem]{Example}
\newtheorem{pro}[lem]{Proposition}
\newtheorem{defi}[lem]{Definition}

\newcommand{\eme}{\mathcal{M}}
\newcommand{\ene}{\mathcal{N}}
\newcommand{\ete}{\mathcal{T}}

\newcommand{\cP}{\mathcal{P}}
\newcommand{\cI}{\mathcal{I}}

\begin{document}

\title[Proper splittings and reduced solutions of matrix equations]{Proper splittings and reduced solutions of matrix equations}

\author{M. Laura Arias}
\address[M. Laura Arias]{Departamento de Matem\'atica, Facultad de Ingenier\'ia, Universidad de Buenos Aires, Argentina}
\address{
and}
\address{
Instituto Argentino de Matem\'atica-CONICET\\
Saavedra 15, Piso 3 (1083), Buenos Aires, Argentina}

\email{lauraarias@conicet.gov.ar}

\author{M. Celeste Gonzalez}
\address[M.Celeste Gonzalez]{Instituto de Ciencias, Universidad Nacional de General Sarmiento, Los Polvorines, Buenos Aires, Argentina}
\address{
and}
\address{
Instituto Argentino de Matem\'atica-CONICET\\
Saavedra 15, Piso 3 (1083), Buenos Aires, Argentina}
\email{celeste.gonzalez@conicet.gov.ar}

\thanks{\emph{Corresponding author.} M. Laura Arias (lauraarias@conicet.gov.ar)}
\thanks{\emph{Acknowledgments.} M. Laura Arias was supported in part by FONCYT (PICT 2017-0883) and UBACyT (20020190100330BA). M. Celeste Gonzalez was supported in part by FONCYT (PICT 2017-0883). }

\subjclass[2010]{15A06, 65F10, 15B48}

\date{}

\keywords{iterative processes, proper splittings, matrix equations}


\begin{abstract} In this article we apply proper splittings of matrices to develop an  iterative process to approximate solutions of matrix equations of the form $TX=W.$ Moreover, by using the partial order
induced by  positive semidefinite matrices, we obtain equivalent conditions to the convergence of this process. We also include some speed comparison results of the convergence of this method.  In addition, for all matrix $T$ we propose a proper splitting based on the polar decomposition of $T$.
\end{abstract}

\maketitle

\section{Introduction}

The theory of splittings of matrices is a useful tool for the construction of iterative methods for solving systems of linear equations. Roughly speaking, a splitting of a matrix $T\in M^{m\times n}(\mathbb{C})$ is a partition of $T$ of the form: $T=U-V$ with  $U, V \in M^{m\times n}(\mathbb{C})$.  
A pioneering work regarding splitting of matrices is due to Varga \cite{MR0158502} where the concept of regular splitting for invertible matrices  $T\in M^{n}(\mathbb{R})$ was introduced  to approximate the unique solution $T^{-1}w$ of the system $Tx=w$ by means of the following iterative process: 
\begin{equation}\label{iterativo-inversible}
x^{i+1}= U^{-1} Vx^{i}+U^{-1} w.
\end{equation}
The scheme (\ref{iterativo-inversible}) converges for every $x^0$ if and only if the spectral radius of $U^{-1}V$  is less than one (in symbols, $\rho(U^{-1} V)<1$) and, in such case, it converges  to $T^{-1}w$. Later, Berman and Plemmons \cite{MR348984} extended the notion of regular splitting for $T\in M^{m\times n}(\mathbb{R})$. They called a proper splitting of $T\in M^{m\times n}(\mathbb{R})$ to a partition $T=U-V$ where $U,V\in M^{m\times n}$, $\mathcal{R}(U)=\mathcal{R}(T)$ and $\mathcal{N}(U)=\mathcal{N}(T)$. Here, $\mathcal{R}(T)$ and $\mathcal{N}(T)$ denote the range and nullspace of $T$, respectively. Based on this concept, Berman and Plemmons proved that the next iterative process
\begin{equation}\label{iterativo-no-inversible}
x^{i+1}= U^\dagger Vx^{i}+U^\dagger w
\end{equation}
converges to the best least square approximate solution $T^\dagger w$ of $Tx=w$  for every $x^0$ if and only if $\rho(U^{\dagger} V)<1$. 

In view of the above, during the last years several authors dealt with the problems of finding conditions on the matrices $U, V$ (of the splitting) which ensure that $\rho(U^{-1} V)<1$ (or, $\rho(U^{\dagger} V)<1$) and of comparing the speed of convergence of these methods for different classes of splittings and matrices. We mention some of these classes:  regular splitting \cite{MR0158502}, \cite{MR745083}, nonnegative splitting with first and second type \cite{MR1113154}, \cite{MR1930390}, weak regular splitting with first and second type \cite{MR1286436}, \cite{MR1628383}, \cite{MR1969059}, proper splitting \cite{MR3141710}, \cite{MR348984}, $P$-regular splitting \cite{MR0075677}, \cite{John}, $P$-proper splitting \cite{MR3671533}, weak nonnegative splitting of the first and the second type \cite{MR1695402},  $T$ monotone \cite{MR0158502}, \cite{MR1286436}, \cite{MR1695402}, $T$ positive definite \cite{MR1368073}, $T$ positive semidefinite  \cite{MR3671533}, among many others.  All these articles are based on working with nonnegative matrices or  positive definite matrices and the usual partial orders induced by them. In the case of considering nonnegative matrices, the results follow by the Perron-Frobenius theory. In \cite{MR2198937}, the authors unified the theory for these two partial orders and for the scheme (\ref{iterativo-inversible})  by using the partial order induced by positivity cone of matrices. However, the theory for the scheme (\ref{iterativo-no-inversible}) requires a deeper analysis. A first approach is given in \cite{MR348984}, where under the idea of proper splitting and considering matrices that leave a cone invariant, many results as those of Varga \cite{MR0158502}, Ortega and Rheinboldt \cite{MR215487}, Mangasarian \cite{MR266409} and Vandergraft \cite{MR309971}  are extended to scheme (\ref{iterativo-no-inversible}). In fact, nonnegative matrices leave the cone $\mathbb{R}_+^n$ invariant and so applying the ideas of \cite{MR348984} many of the results about scheme (\ref{iterativo-inversible}) and nonnegative matrices are immediately extended for (\ref{iterativo-no-inversible}). However, these ideas can not be applied when the partial order induced by  positive semidefinite matrices, also known as L\"owner order,  is considered. One of our goal in this article is to fill this gap. Moreover, we observe that proper splitting can also be applied to solve matrix equations of the form $TX=W$ and to get their so-called reduced solutions. This sort of solutions emerges as a generalization of the well-known Douglas'solutions \cite{MR0203464} and are useful, for example, in many problems of engineering and physic in which it is needed to separate signals from noise (see \cite{286957}, \cite{MR2512471}, among others).  In this article, given a proper splitting $T=U-V,$ we propose the iterative method (\ref{proceso iterativo})  and we prove that it converges to a reduced solution of $TX=W$ if and only if $\rho(U^\dagger V)<1.$ In addition, we provide equivalent conditions for  (\ref{proceso iterativo}) to converge to an Hermitian or positive definite reduced solution. Motivated by all these facts, we devote the remain of this paper to provide, for a proper splitting $T=U-V$, equivalent conditions to  $\rho(U^\dagger V)<1$  and some speed comparison results, considering the L\"owner order. In addition, we define and study a proper splitting which is based on the polar decomposition of matrices. 

We briefly describe the contents of the paper. In Section \ref{preliminaries}  we fix the notation and give some results that are needed in the following sections. In Section \ref{redsolutions}  we study some spectral properties of the so-called reduced solutions of matrix equations $TX=W$ and we analyse the  existence of Hermitian and positive semidefinite reduced solutions. In Section \ref{propersplittingmatrices}  we study some properties of proper splitting that will be used in the next sections. Section \ref{RsandProper} is entirely devoted to explore the relationship between proper splitting and reduced solutions. We propose the iterative process (\ref{proceso iterativo}) with the aim of approximate the reduced solutions of $TX=W$ by means of the use of a proper splitting of $T$. In addition, we analyse the convergence of this iterative process to Hermitian and positive definite solutions. Here the main result is Theorem \ref{convergencia}. In Section  \ref{convergenciaps}  we present our main results regarding convergence and speed comparison for the iterative method (\ref{proceso iterativo}) when the L\"owner order is considered. The main results in this section are Propositions \ref{equiTdaggerV}  and \ref{H-rho-CS} and Theorems \ref{comp1} and \ref{comp2}. Finally, in Section \ref{Sec-Ejemplos} we present a particular proper splitting  for $T\in\mathbb{M}^{m\times n} (\mathbb{C})$ based on the polar decomposition of $T$. In Proposition \ref{polar splitting} we analyse the convergence of this particular splitting. In Theorem \ref{PP} and Proposition \ref{normales} we study this splitting for the class of matrices  $T\in M^{n\times n}(\mathbb{C})$ that can be factorized as the product of two orthogonal projections and for normal matrices.

\section{Preliminaries}\label{preliminaries}

Along this article the set of $n$-uples of complex numbers is denoted by $\mathbb{C}^n$  and $M^{m\times n}(\mathbb{C})$ denotes the set of $m\times n$-matrices with coefficients in $\mathbb{C}$. If $m=n$ we abbreviate $M^{n}(\mathbb{C})$. Throughout, we shall consider the canonical inner product, $\pint{ \ , \ }$, on $\mathbb{C}^n$ and $\|\cdot\|=\pint{\cdot,\cdot}^{1/2}$ the norm induced by it. The spectral norm in $M^{n}(\mathbb{C})$ will be denoted by $\| \ \|$. Given $T\in M^{m\times n}(\mathbb{C})$ the symbols $\sigma(T)$ and $\rho(T)$ stand for the spectrum (set of eigenvalues) of $T$ and spectral radius of $T$, respectively. In addition, given $T\in M^{m\times n}(\mathbb{C})$ we denote by $\mathcal{R}(T)$,  $\mathcal{N}(T)$, $r(T)$ , $T^*$, the range, the nullspace, the rank  and complex conjugate transpose of $T$, respectively.  The Moore Penrose inverse of $T\in M^{m\times n}(\mathbb{C})$ is denoted by $T^\dagger$ and it is the unique matrix in $M^{n\times m}(\mathbb{C})$ which satisfies the four ``Moore-Penrose equations''
\begin{equation}\label{mp-equations}
TXT=T, \  \   XTX=X, \  \   XT=(XT)^*, \  \ TX=(TX)^*.
\end{equation}

The next classical result of Greville \cite{MR210720} related to the reverse order law for the Moore-Penrose inverse will be used along the article. 

\begin{thm}\label{Greville}
Let $T_1,T_2\in M^n(\mathbb{C})$. Then, $(T_1 T_2)^\dagger = T_2^\dagger T_1^\dagger$ if and only if $\mathcal{R}(T_1^*T_1T_2)\subseteq \mathcal{R}(T_2)$ and $\mathcal{R}(T_2T_2^*T_1^*)\subseteq \mathcal{R}(T_1^*)$. 
\end{thm}

Henceforth, we denote by $M^n_H(\mathbb{C})$ the set of Hermitian matrices of $M^n(\mathbb{C})$ and by $M_+^n(\mathbb{C})$ the set of positive semidefinite matrices, i.e.,  $T\in M_+^n(\mathbb{C})$ if $\pint{Tx,x}\geq 0$ for all $x\in\mathbb{C}^n$. Moreover, given $T_1,T_2\in M^n_H(\mathbb{C})$, we say that $T_1\leq T_2$ if $T_2-T_1\in M_+^n(\mathbb{C})$, i.e., $\leq$ denotes the L\"owner order of matrices.

The following result  will be frequently used in these notes. The proof  can be found in \cite[Corollary 2]{MR1048801}.

\begin{thm}\label{lowner-mp}Let $T_1,T_2\in M_+^n(\mathbb{C})$ and consider the following conditions 
\begin{enumerate}
\item $T_1\leq T_2$;
\item $T_2^\dagger \leq T_1^\dagger$;
\item[$c_1)$] $r(T_1)=r(T_2)$;
\item[$c_2)$] $\mathcal{R}(T_2-T_1)\cap\mathcal{N}(T_2)=\mathcal{R}(T_1^\dagger-T_2^\dagger)\cap\mathcal{N}(T_1)=\{0\}$.
\end{enumerate}
Then any of two the conditions $a), b)$ and $ c_i)$ imply the third condition, $i=1,2$.
\end{thm}


\section{Reduced solutions of matrix equations}\label{redsolutions}

In this section we focus on the study of matrix equations of the form $TX=W$ for $T\in M^{m\times n}(\mathbb{C})$ and $W\in M^{m\times r}(\mathbb{C})$. More precisely, we are interested on the so-called ``reduced solutions'' of these equations introduced in the next result.

\begin{thm}\label{Douglas}
Let $T\in M^{m\times n}(\mathbb{C})$ and $W\in M^{m\times r}(\mathbb{C})$.  There exists $X_0\in M^{n\times r}(\mathbb{C})$ such that $TX_0=W$ if and only if $\mathcal{R}(W)\subseteq \mathcal{R}(T).$ In such case, for each subspace $\eme$ such that $\eme\overset{.}{+}\mathcal{N}(T)=\mathbb{C}^n$ there exists a unique $X_\eme\in  M^{n\times r}(\mathbb{C}) $ such that $TX_\eme=W$ and $\mathcal{R}(X_\eme)\subseteq \eme.$ Moreover, $\mathcal{N}(X_\eme)=\mathcal{N}(W)$. We called $X_\eme$ the {\bf reduced solution} for $\eme$ of $TX=W.$ In particular, it holds that $X_{\mathcal{R}(T^*)}=T^\dagger W.$
\end{thm}
\begin{proof}
\begin{enumerate}
See  \cite{MR2407082} and \cite{MR0203464}.
\end{enumerate}
\end{proof}

The previous result also holds for $T,W$ and $X_0$ bounded linear operators on Hilbert spaces, see \cite{MR0203464}. 

\begin{rem}\label{RemRed} In this remark we comment some advantages of working with reduced solutions:
\begin{enumerate}
\item Reduced solutions and oblique projections are two concepts closely related. Indeed, the set of oblique projections with a fixed nullspace $\ete$ coincides with the set of reduced solutions of the equation $(I-P_\ete)X=I-P_\ete,$ where $P_\ete$ denotes the orthogonal projection onto $\ete.$ In a more general context, given $T, W\in M^{m\times n}(\mathbb{C})$ such that $\mathcal{R}(W)\subseteq \mathcal{R}(T)$ and $\eme,\ene$ two algebraic complements of $\mathcal{N}(T)$, then it is easy to verify that $X_\ene=Q_{\ene//\mathcal{N}(T)}X_\eme$  where $Q_{\ene//\mathcal{N}(T)}$ is the projection with range $\ene$ and nullspace $\mathcal{N}(T)$. This relationship between reduced solutions will be frequently used along the article.

Oblique projections (and so reduced solutions) appear in several problems of engineering, physic, chemistry, among other areas. For instance, they are useful for solving a variety of signal processing problems where it is needed to separate signals from noise (see \cite{286957},\cite{MR2512471}, \cite{BOYER20092547}, among other sources). Furthermore, oblique projections are used in communication problems  to  remove intersymbol interference \cite{286957} and for estimation of directions-of-arrival \cite{MR2512471}.


\item  Other class of matrix equations that frequently arises in many applications are those of the form $TXS=W$, with $T\in M^{m\times n}(\mathbb{C})$ and  $S\in M^{r\times s}(\mathbb{C}).$ Let us observe that this class of equations can be transformed into two-steps matrix equations of the previous form by means of reduced solutions. More precisely, if $Y_\eme$ is a reduced solution of $TX=W$ then $\mathcal{N}(Y_\eme)=\mathcal{N}(W)$ or, equivalently, $\mathcal{R}(Y_\eme^*)=\mathcal{R}(W^*)$ and so $\mathcal{R}(Y_\eme^*)\subseteq \mathcal{R}(S^*).$ Thus, there exists $Y^*$ such that $S^*Y^*=Y_\eme^*$. Hence, $TYS=TY_\eme=W,$ i.e., $Y$ solves $TXS=W.$ We shall return to this idea in Section \ref{RsandProper}.
\item Reduced solutions can also be applied to fully describe the generalized inverses of a matrix $T\in M^{m\times n}(\mathbb{C})$, see \cite{MR2407082}. For instance, $T^\dagger$ coincides with the reduced solution of $TX=P_{R(T)}$ for $\eme=R(T^*).$ 
\end{enumerate}
\end{rem}

\begin{lem}\label{rhoreducidas} Let $T, W\in M^{m\times n}(\mathbb{C})$ such that $\mathcal{R}(W)\subseteq \mathcal{R}(T)$ and $\mathcal{N}(T)\subseteq \mathcal{N}(W).$  Then the following assertions hold:
\begin{enumerate}
\item $\sigma (X_\eme)=\sigma(T^\dagger W)$ for all reduced solution $X_\eme$ of $TX=W$.
\item $\rho(X_\eme)=\rho(T^\dagger W)$ for all reduced solution $X_\eme$ of $TX=W$.
\item Let $X_\eme$ be the reduced solution for $\eme$ of $TX=W$ and $\lambda\in \sigma(X_\eme)$. Hence, if $v_\lambda$ is an eigenvector of $X_\mathcal{M}$ associated to  $\lambda$ then $Q_{\mathcal{N}// \mathcal{N}(T)}v_\lambda$ is an eigenvector of $X_{\mathcal{N}}$ associated to $\lambda$ for all reduced solution $X_{\mathcal{N}}$ of $TX=W$.  
\end{enumerate}
\end{lem}
\begin{proof} \ 

$a)$  Since $X_\eme=Q_{\eme//\mathcal{N}(T)} T^\dagger W$ we get that $\sigma(X_\eme)=\sigma(Q_{\eme//\mathcal{N}(T)} T^\dagger W)=\sigma(T^\dagger W Q_{\eme//\mathcal{N}(T)})=\sigma(T^\dagger W)$, since $\mathcal{N}(T)\subseteq \mathcal{N}(W).$ .

$b)$ It follows from the above item.

$c)$ Again, it follows from the fact that $ X_{\mathcal{N}}=Q_{\mathcal{N}//\mathcal{N}(T)}X_\eme.$

\end{proof}

As we have already said, matrix equations of the form $TX=W$ arise in many problems of engineering, statistics, physics, among others areas. However, in many of these applications the only relevant solutions are those where the solution matrix is Hermitian or positive semidefinite. This occurs, for example, in many statistical problems (see \cite{RAOCHAGANTY1997421}, \cite{MATHEW1997129}, just to mention a few). Evidently, non every solvable matrix equation $TX=W$ admits a  positive semidefinite nor Hermitian solution. The following result due to Khatri and Mitra \cite{MR417212} provides equivalent conditions to the existence of such kind of solutions.

\begin{thm}\label{soluciones}
Let $T,W\in M^{m\times n}(\mathbb{C}).$ The following assertions hold:
\begin{enumerate}
\item If $\mathcal{R}(W)\subseteq \mathcal{R}(T)$ then there exists $X_0\in M^n_H(\mathbb{C})$ such that $TX_0=W$ if and only if $TW^*\in M^m_H(\mathbb{C}).$
\item There exists $X_0\in M^n_+(\mathbb{C})$ such that $TX_0=W$ if and only $TW^*\in M^m_+(\mathbb{C})$ and $r(TW^*)=r(W).$
\end{enumerate}
\end{thm}

As we mentioned before, we are interested on the reduced solutions of the matrix equation $TX=W$.  In the following results we study Hermitian and positive semidefinite reduced solutions:

\begin{pro}\label{exredhermitiana}
Let $T,W\in M^{m\times n}(\mathbb{C})$ such that $\mathcal{R}(W)\subseteq \mathcal{R}(T)$. The following conditions are equivalent:
\begin{enumerate}
\item There exists an Hermitian reduced solution of $TX=W.$
\item $TW^*\in M^m_H(\mathbb{C})$ and $\mathcal{R}(W^*)\cap\mathcal{N}(T)=\{0\}.$
\end{enumerate}
\end{pro}
\begin{proof}
$a)\rightarrow b)$ Let $X_\eme$ be an Hermitian reduced solution of $TX=W.$ Hence, as $\mathcal{N}(X_\eme)=\mathcal{N}(W)$, we get that $\mathcal{R}(X_\eme)=\mathcal{R}(W^*)$ and so $\mathcal{R}(W^*)\cap\mathcal{N}(T)=\{0\}.$ The proof is complete applying Theorem \ref{soluciones}.

$b)\rightarrow a)$ By Theorem \ref{soluciones}, there exists $X_0 \in M^n_H(\mathbb{C})$ such that $TX_0=W.$ On the other hand,   since $\mathcal{R}(W^*)\cap\mathcal{N}(T)=\{0\}$, then $\eme:=\mathcal{R}(W^*)+(\mathcal{R}(W^*)+\mathcal{N}(T))^\bot$ is an algebraic complement of $\mathcal{N}(T)$ in $\mathbb{C}^n.$ Therefore, the projection $Q_{\eme//\mathcal{N}(T)}$ is well-defined and $TQ_{\eme//\mathcal{N}(T)}X_0 Q_{\eme//\mathcal{N}(T)}^*=W Q_{\eme//\mathcal{N}(T)}^*=W$ because $\mathcal{R}(W^*)\subseteq \eme$. Thus, by the uniqueness of the reduced solution for $\eme,$ we get that $Q_{\eme//\mathcal{N}(T)}X_0 Q_{\eme//\mathcal{N}(T)}^* \in M^n_H(\mathbb{C})$ is the reduced solution for $\eme$ of $TX=W$. The proof is conclude.

\end{proof}

 If $T$ and $W$ are bounded operators defined on Hilbert spaces, then the condition `` $\overline{\mathcal{R}(W^*)}\overset{.}{+}\mathcal{N}(T)$ is a closed subspace'' should be added in item b) of 
Proposition \ref{exredhermitiana} in order to get the result. This condition is equivalent to an angle condition between the susbpaces $\overline{\mathcal{R}(W^*)}$ and $\mathcal{N}(T)$, see \cite{MR2514477}.

\begin{pro}\label{redhermitiana} Let $T,W\in M^{m\times n}(\mathbb{C})$ such that $\mathcal{R}(W)\subseteq \mathcal{R}(T)$ and let $\eme$ be a subspace of $\mathbb{C}^n$ such that $\eme\overset{.}{+}\mathcal{N}(T)=\mathbb{C}^n$. If  $TW^*\in M^m_H(\mathbb{C})$ and $X_{\mathcal{M}}$ is the reduced solution for $\mathcal{M}$ of $TX=W$ then the following conditions are equivalent:
\begin{enumerate}
\item $X_{\mathcal{M}}$ is Hermitian;
\item $X_{\mathcal{M}}$ is normal;
\item $\mathcal{R}(W^*)\subseteq \eme$.

\end{enumerate}
In particular,  $T^\dagger W\in M^n_H(\mathbb{C})$ if and only if $T^\dagger W$ is normal   if and only if $\mathcal{R}(W^*)\subseteq \mathcal{R}(T^*)$.
\end{pro}

\begin{proof} $ \ $
$a)\rightarrow b)$ It is immediate.

$b)\rightarrow c)$ If $X_{\mathcal{M}}X^*_{\mathcal{M}}=X^*_{\mathcal{M}}X_{\mathcal{M}}$ then $\mathcal{R}(X_{\mathcal{M}})=\mathcal{R}(X_{\mathcal{M}}^*).$ So that $\mathcal{R}(W^*)= \mathcal{N}(W)^\bot=\mathcal{N}(X_{\mathcal{M}})^{\bot}=\mathcal{R}(X_{\mathcal{M}})\subseteq \mathcal{M}$.

$c)\rightarrow a)$ By Theorem \ref{soluciones}, there exists $X_0\in M^n_H(\mathbb{C})$ such that $TX_0=W.$ Therefore, $TQ_{\eme//\mathcal{N}(T)}X_0 Q_{\eme//\mathcal{N}(T)}^*=W Q_{\eme//\mathcal{N}(T)}^*=W$ because $\mathcal{R}(W^*)\subseteq \eme$. Thus, by the uniqueness of the reduced solution for $\eme,$ we get that $X_\eme=Q_{\eme//\mathcal{N}(T)}X_0 Q_{\eme//\mathcal{N}(T)}^*,$ i.e, $X_\eme\in M^n_H(\mathbb{C})$.

\end{proof}

\begin{cor}
Let $T,W\in M^{m\times n}(\mathbb{C})$ such that $\mathcal{R}(W)\subseteq \mathcal{R}(T)$. If $T^\dagger W\in M^n_H(\mathbb{C})$ then $\rho(X_{\mathcal{M}})=\rho(T^\dagger W)$ for all reduced solution for $\mathcal{M}$ of $TX=W$.
\end{cor}

\begin{proof}
It follows from Lemma \ref{rhoreducidas} and Proposition \ref{redhermitiana}.
\end{proof}

We omit the proofs of the next results concerning  positive semidefinite reduced solutions of $TX=W$ since they follow proceeding as in the proofs of Propositions \ref{exredhermitiana} and \ref{redhermitiana}.

\begin{pro}\label{exredpositiva}
Let $T,W\in M^{m\times n}(\mathbb{C})$ such that $\mathcal{R}(W)\subseteq \mathcal{R}(T)$. The following conditions are equivalent:
\begin{enumerate}
\item There exists a  positive semidefinite reduced solution of $TX=W.$
\item $TW^*\in M^m_+(\mathbb{C})$, $r(TW^*)=r(W)$ and $\mathcal{R}(W^*)\cap\mathcal{N}(T)=\{0\}.$
\end{enumerate}
\end{pro}

\begin{pro}\label{redpositiva} Let $T,W\in M^{m\times n}(\mathbb{C})$ such that $TW^*\in M^n_+(\mathbb{C})$ and $r(TW^*)=r(W)$, and let $\eme$ be a subspace of $\mathbb{C}^n$ such that $\eme\overset{.}{+}\mathcal{N}(T)=\mathbb{C}^n$. If  $X_{\mathcal{M}}$ is the reduced solution for $\mathcal{M}$ of $TX=W$ then the following conditions are equivalent:
\begin{enumerate}
\item $X_{\mathcal{M}}$ is  positive semidefinite;
\item $X_{\mathcal{M}}$ is normal;
\item $\mathcal{R}(W^*)\subseteq \eme$.

\end{enumerate}
In particular, $T^\dagger W\in M^n_+(\mathbb{C})$  if and only if $T^\dagger W$ is normal if and only if $\mathcal{R}(W^*)\subseteq \mathcal{R}(T^*)$.
\end{pro}


\section{Proper splitting of matrices}\label{propersplittingmatrices}

Given a matrix $T\in M^{m\times n}(\mathbb{C})$, a splitting of $T$ is a partition of $T$ in the following form: $T=U-V$ with  $U, V \in M^{m\times n}(\mathbb{C})$. 
 In \cite{MR348984}, Berman and Plemmons introduced the concept of proper splitting for rectangular matrices with the aim of developing an iterative process that converges to the best least squares approximate solution of the system $Tx=w.$ Let us formally introduce this concept:

\begin{defi}
Let $T\in M^{m\times n}(\mathbb{C})$. The decomposition
$$
T=U-V,
$$
where $U, V \in M^{m\times n}(\mathbb{C})$ is called a proper splitting of $T$ if  $\mathcal{R}(U)=\mathcal{R}(T)$ and $\mathcal{N}(U)=\mathcal{N}(T)$.
\end{defi}

In the following result we gather some properties of proper splittings. Many of these properties can be found in \cite{MR348984}.

\begin{pro}\label{propiedades splitting}
If $T=U-V$ is a proper splitting of $T$ then the following assertions hold:
\begin{enumerate}
\item $\mathcal{N}(T)\subseteq \mathcal{N}(V)$ and $\mathcal{R}(V)\subseteq \mathcal{R}(T)$.
\item $T=U(I-U^\dagger V)$ and $I-U^\dagger V$ is invertible.
\item $T^\dagger= (I-U^\dagger V)^{-1} U^\dagger$.
\item $(T^\dagger U)^\dagger = U^\dagger T$.
\end{enumerate}
\end{pro}

\begin{proof} Let $T=U-V$ be a proper splitting of $T$.

$a)$ If $x\in\mathcal{N}(T)$ then $Tx=0$ and so $Ux=0$. Then $Vx=0$. Therefore  $\mathcal{N}(T)\subseteq \mathcal{N}(V)$. In addition, since $V=U-T$ and $\mathcal{R}(U)=\mathcal{R}(T)$ then $\mathcal{R}(V)\subseteq \mathcal{R}(T)$.

$b), c)$ The proofs can be found in \cite[Theorem 1]{MR348984}.

$d)$  Since $\mathcal{R}(U)=\mathcal{R}(T)$ and $\mathcal{N}(U)=\mathcal{N}(T)$ the assertion follows from Theorem \ref{Greville}.
\end{proof}

The following results will be useful in the next section.

\begin{pro}\label{UdaggerV-Hermitiana}
Let $T=U-V$ be  a proper splitting of $T$. The following assertions are equivalent:
\begin{enumerate}
\item $VT^*\in M^n_H(\mathbb{C});$
\item $UT^*\in M^n_H(\mathbb{C});$
\item $VU^*\in M^n_H(\mathbb{C});$
\item $U^\dagger V\in M^n_H(\mathbb{C})$;
\item $T^\dagger U\in M^n_H(\mathbb{C})$;
\item $T^\dagger V \in M^n_H(\mathbb{C})$.
\end{enumerate}
\end{pro}

\begin{proof} $ \ $

$a)\leftrightarrow b)$ It is obvious since $UT^*=TT^*+VT^*$. 

$b)\leftrightarrow c)$ It is obvious since $VU^*=UU^*-TU^*.$ 

The equivalences $c)\leftrightarrow d)$, $b)\leftrightarrow e)$  and $a)\leftrightarrow f)$ follow from Proposition \ref{redhermitiana} and Theorem \ref{soluciones}.
\end{proof}

\begin{pro}\label{UdaggerV positiva}
Let $T=U-V$ be  a proper splitting of $T$. The following assertions are equivalent:
\begin{enumerate}
\item $U^\dagger V\in M^n_+(\mathbb{C})$;
\item $VU^*\in M^n_+(\mathbb{C})$ and $r(VU^*)=r(V);$
\item $U^\dagger T\in M^n_H(\mathbb{C})$ and $U^\dagger T \leq P_{T^*}$.
\end{enumerate}
\end{pro}

\begin{proof} $ \ $

$a)\leftrightarrow b)$ It follows by Theorem \ref{soluciones} and Proposition \ref{redpositiva}.

$a)\leftrightarrow c)$ Suppose that $U^\dagger V\in M^n_+(\mathbb{C})$. Then $U^\dagger T=U^\dagger U-U^\dagger V\in M^n_H(\mathbb{C})$. In addition, $U^\dagger V=U^\dagger U-U^\dagger T\geq 0$. Therefore, $U^\dagger T\leq P_{T^*}$.  Conversely,  if $U^\dagger T\in M^n_H(\mathbb{C})$ and $U^\dagger T \leq P_{T^*}=U^\dagger U$ then $U^\dagger V=U^\dagger U-U^\dagger T \geq 0$. 
\end{proof}

\begin{pro}\label{TU}
Let $T=U-V$ be  a proper splitting of $T$. The following assertions are equivalent:
\begin{enumerate}
\item $U^\dagger T\in M^n_+(\mathbb{C})$;
\item $TU^*\in M^n_+(\mathbb{C})$ and $r(TU^*)=r(T);$
\item  $U^\dagger V\in M^n_H(\mathbb{C})$ and $U^\dagger V\leq P_{T^*}$.
\end{enumerate}
\end{pro}

\begin{proof} $ \ $

$a)\leftrightarrow b)$ It follows by Theorem \ref{soluciones} and Proposition \ref{redpositiva}.

$a)\leftrightarrow c)$ If $U^\dagger T\in M^n_+(\mathbb{C})$ then $U^\dagger V=U^\dagger U-U^\dagger T\in M^n_H (\mathbb{C})$. Furthermore $U^\dagger T=U^\dagger U-U^\dagger V\geq 0$. Then $P_{T^*}\geq U^\dagger V$.  Conversely, if  $U^\dagger V\in M^n_H(\mathbb{C})$ and $P_{T^*}\geq U^\dagger V$ then $U^\dagger T=U^\dagger U-U^\dagger V\geq 0$. 
\end{proof}


\section{Reduced solutions and proper splittings of matrices}\label{RsandProper}

In \cite{MR348984} the concept of proper splitting of a matrix $T$ is applied to approximate the best minimum square solution of a system $Tx=w.$ In this section, we shall observe that this class of splitting can also be used to get the reduced solution for $\eme$ of a matrix equation $TX=W.$ To be more specific, let $T\in M^{m\times n}(\mathbb{C})$ and $W\in M^{m\times r}(\mathbb{C})$  such that $\mathcal{R}(W)\subseteq \mathcal{R}(T)$ and let $\eme$ be a subspace of $\mathbb{C}^n$ such that $\eme\overset{.}{+}\mathcal{N}(T)=\mathbb{C}^n$. Let $T=U-V$ be a proper splitting of $T$. Let  $Y_\eme\in M^{n}(\mathbb{C})$ and $Z_\eme \in M^{n\times r}(\mathbb{C})$ be the reduced solutions for $\eme$ of $UY=V$ and $UZ=W,$ respectively. Define the iterative process:
\begin{equation}
X^{i+1}=Y_\eme X^i+ Z_\eme  .\label{proceso iterativo}
\end{equation}
We call (\ref{proceso iterativo}), {\it{the iterative process for $\eme$ of the proper splitting $T=U-V$ with respect to $W.$}}  

In the next theorem,  we prove that whether the iteration (\ref{proceso iterativo}) converges, then it converges  to the reduced solution for $\eme$ of $TX=W.$ Furthermore, we provide an equivalent condition for its convergence.

\begin{thm}\label{convergencia}
Let $T\in M^{m\times n}(\mathbb{C})$ and $W\in M^{m\times r}(\mathbb{C})$ such that $\mathcal{R}(W)\subseteq \mathcal{R}(T)$ and let $\eme$ be a subspace of $\mathbb{C}^n$ such that $\eme\overset{.}{+}\mathcal{N}(T)=\mathbb{C}^n$.  Consider the proper splitting $T= U-V$ of $T$. Then the iterative process (\ref{proceso iterativo}) is convergent for all $X^0$ if and only if $\rho(Y_\eme)<1$. In such case,  the limit of the iterative process (\ref{proceso iterativo}) coincides with $X_\eme$, the reduced solution for $\eme$ of the equation  $TX=W.$ 
\end{thm}
\begin{proof} It is clear that if the iterative process (\ref{proceso iterativo}) converges to $X\in M^{n\times r}(\mathbb{C})$ then $TX=W$ and $\mathcal{R}(X)\subseteq \eme,$ i.e., $X=X_\eme.$ Now,  $UX_\eme=VX_\eme+W$ and so $U(X^{i+1}-X_\eme)=V(X^{i}-X_\eme)=UY_\eme(X^{i}-X_\eme).$ Then, $X^{i+1}-X_\eme=Y_\eme(X^{i}-X_\eme)=Y_\eme^{i+1}(X^0-X_\eme)$ since $\eme\overset{.}{+}\mathcal{N}(U)=\mathbb{C}^n.$ Thus, the iterative process (\ref{proceso iterativo}) is convergent for all $X^0$ if and only if $\rho(Y_\eme)<1$.
\end{proof}

It should be mentioned that the iterative process (\ref{proceso iterativo}) emerges as a generalization of the iteration introduced in \cite{MR348984} to obtain the best minimum square solution of $Tx=b.$ Indeed,  notice that choosing $W=P_T b$ and $\eme=R(T^*)$ then the iteration (\ref{proceso iterativo}) can be used to obtain the best minimum square error solution of $Tx=b.$
Furthermore, we stress that the iteration (\ref{proceso iterativo}) is practical when it is easier to solve the equations that involve $U$ than the equation that involves $T$. In Section \ref{Sec-Ejemplos} we propose some convenient proper splitting. 

\begin{rems}
\begin{enumerate}
\item Taking into account Remark \ref{RemRed}, we observe that the iterative process (\ref{proceso iterativo}) can also be used for solving matrix equations of the form $TXS=W.$ Evidently, in this case we need a proper splitting for $T$ and a proper splitting for $S$. We recommend \cite{KHORSANDZAK2013269} for other applications of splitting of matrices for solving matrix equations of the form $TXS=W.$
\item As it was pointed out in Remark \ref{RemRed}, the generalized inverses of a matrix $T$ can be described by means of reduced solutions of certain matrix equations.  Therefore, the iterative process (\ref{proceso iterativo}) can be applied to approximate generalized inverses of a matrix. We recommend \cite{MR2115593} and \cite{MR2490784} for other perspective of the use of splitting of matrices or operators to get generalized inverses.  
\end{enumerate}
\end{rems}

By Lemma \ref{rhoreducidas}, Proposition \ref{propiedades splitting} and Theorem \ref{convergencia} we get the next result:

\begin{cor}\label{corrho} Let $T= U-V$ be a proper splitting of $T\in M^{m\times n}(\mathbb{C})$ and let $W\in M^{m\times r}(\mathbb{C})$ such that $\mathcal{R}(W)\subseteq \mathcal{R}(T)$. Then, the iterative process of the proper splitting $T=U-V$ converges for some $\eme$ such that $\eme\overset{.}{+}\mathcal{N}(T)=\mathbb{C}^n$ if and only if    the iterative process of the proper splitting $T=U-V$ converges for all $\eme$ such that $\eme\overset{.}{+}\mathcal{N}(T)=\mathbb{C}^n.$ In particular, the iterative process (\ref{proceso iterativo}) is convergent if and only if $\rho(U^\dagger V)<1.$
\end{cor}

The following result follows from Propositions \ref{redhermitiana}, \ref{redpositiva} and Theorem \ref{convergencia}.

\begin{cor} Let $T,W\in M^{m\times n}(\mathbb{C})$ such that $\mathcal{R}(W)\subseteq \mathcal{R}(T)$ and let $\eme$ be a subspace of $\mathbb{C}^n$ such that  $\eme\overset{.}{+}\mathcal{N}(T)=\mathbb{C}^n$. Consider the iterative process for $\eme$ of the proper splitting $T=U-V$ with respect to $W,$ given by (\ref{proceso iterativo}). Then the following assertions hold:
\begin{enumerate}
\item If $WT^*\in M^m_H(\mathbb{C})$ and $\rho(Y_\eme)<1$ then the iterative process (\ref{proceso iterativo}) converges to an Hermitian reduced solution of $TX=W$ if and only if $\mathcal{R}(W^*)\subseteq \eme$.
\item If $WT^*\in M^m_H(\mathbb{C})$, $r(WT^*)=r(W)$ and $\rho(Y_\eme)<1$ then the iterative process (\ref{proceso iterativo}) converges to a positive semidefinite reduced solution of $TX=W$ if and only if $\mathcal{R}(W^*)\subseteq \eme$.
\end{enumerate}
\end{cor}


\section{Convergence of proper splittings}\label{convergenciaps}

In the previous section, we defined the iterative process (\ref{proceso iterativo}) on the basis of a proper splitting  $T=U-V$ and we proved that this iterative process is convergent (in which case it converges to the reduced solution for $\eme$ of $TX=W$) if and only if  $\rho(U^\dagger V)<1.$ Accordingly, this section is devoted to provide conditions that guarantee  $\rho(U^\dagger V)<1.$ Most of our results are inspired in \cite{MR348984}, \cite{MR1695402}, \cite{MR1113154}. In order to compare our results with those of the previously mentioned articles, in what follows:
\begin{itemize}
\item by $K$  we denote a full cone in $\mathbb{R}^n$ and given $A\in M^{m\times n}(\mathbb{R})$, by $A\in \cI_K$ we mean that $A$ leaves the full cone $K$ invariant, i.e., $AK\subset K;$ 
\item by $A\in M_+^n(\mathbb{C})$ we  mean $A\geq 0$ where $\geq$ denotes the L\"owner order in $M^n(\mathbb{C});$ 
\item by $A\in M_{\succ 0}^n(\mathbb{R}) $ we mean that $A$ has nonnegative entries and by $\succ$ we denote  the usual order in $M^{m\times n}(\mathbb{R})$ induced by nonnegative matrices.
\end{itemize}

Evidently, $ M_{\succ 0}^n(\mathbb{R})= \cI_K$ with $K=\mathbb{R}^n_+.$ In this section, we work with the L\"owner order. Theorems \ref{Greville} and \ref{lowner-mp} play a key role for proving the main results of this section.

\begin{pro}\label{equiTdaggerV}
Let $T=U-V$ be a proper splitting of $T$. Then the following assertions are equivalent: 
\begin{enumerate}
\item $T^\dagger V\in M^n_+(\mathbb{C})$;
\item $TV^*\in M^n_+(\mathbb{C})$ and $r(TV^*)=r(V);$
\item $0\leq U^\dagger V \leq P_{V^*}$;
\item $U^\dagger V \in M^n_+(\mathbb{C})$ and $\rho(U^\dagger V)=\displaystyle\frac{\rho(T^\dagger V)}{1+\rho(T^\dagger V)}<1$.
\end{enumerate}
\end{pro}

\begin{proof} $ \ $

$a) \leftrightarrow b)$ It follows from Theorem \ref{soluciones} and Proposition \ref{redpositiva}.

$a) \leftrightarrow c)$ Let $T=U-V$ be a proper splitting of $T$. Since $T^\dagger =(I-U^\dagger V)^{-1}U^\dagger$ then
\begin{eqnarray}
0\leq T^\dagger V &\leftrightarrow& 0\leq (I-U^\dagger V)^{-1}U^\dagger V \leftrightarrow 0 \leq \left((I-U^\dagger V)^{-1}U^\dagger V \right)^\dagger  \nonumber \\
&\leftrightarrow& 0\leq \left(U^\dagger V \right)^\dagger (I-U^\dagger V) \leftrightarrow  0\leq \left(U^\dagger V \right)^\dagger - P_{\mathcal{N}(U^\dagger V)^\bot}\nonumber\\
&\leftrightarrow& 0\leq P_{V^*}\leq \left(U^\dagger V\right)^\dagger. \label{TdaggerV}
\end{eqnarray}
 Now, if $0\leq P_{V^*}\leq \left(U^\dagger V\right)^\dagger$ then $\mathcal{R}(P_{V^*})\subseteq \mathcal{R}((U^\dagger V)^\dagger)$. So,   $\mathcal{R}((U^\dagger V)^\dagger - P_{V^*})\cap \mathcal{N}((U^\dagger V)^\dagger)\subseteq \mathcal{R}((U^\dagger V)^\dagger)\cap\mathcal{N}((U^\dagger V)^\dagger)=\{0\}$. On the other hand,  if $0\leq P_{V^*}\leq \left(U^\dagger V\right)^\dagger$ then $\mathcal{R}(P_{V^*})\subseteq \mathcal{R}((U^\dagger V)^\dagger)=\mathcal{R}(U^\dagger V)$. Now, $\mathcal{R}(P_{V^*}-U^\dagger V)\cap\mathcal{N}(P_{V^*})\subseteq \mathcal{R}(U^\dagger V)\cap\mathcal{N}(V)=\mathcal{R}(V^*(U^*)^\dagger)\cap\mathcal{N}(V)\subseteq \mathcal{R}(V^*)\cap\mathcal{N}(V)=\{0\}$. So that, by Theorem \ref{lowner-mp}, we get $0\leq U^\dagger V \leq P_{V^*}$. Conversely, if $0\leq U^\dagger V \leq P_{V^*}$ then $\mathcal{R}((U^\dagger V)^\dagger)=\mathcal{R}(U^\dagger V)\subseteq \mathcal{R}(V^*)$. In addition, as $\mathcal{N}(U^\dagger V)=
\mathcal{N}(V)$ then it holds that $\mathcal{R}(P_{V^*}-U^\dagger V)\cap\mathcal{N}(P_{V^*})=\{0\}$ and $\mathcal{R}((U^\dagger V)^\dagger -P_{V^*})\cap\mathcal{N}(U^\dagger V)=\{0\}$. Therefore, applying again Theorem \ref{lowner-mp}, $0\leq P_{V^*}\leq \left(U^\dagger V\right)^\dagger$. Then, by the equivalences (\ref{TdaggerV}), the assertion follows.

$a)\leftrightarrow d)$ First, let us prove that $\lambda\in \sigma(T^\dagger V)$ if and only if $\mu=\frac{\lambda}{1+\lambda}\in \sigma(U^\dagger V)$. In fact, 
\begin{eqnarray}
T^\dagger Vx=\lambda x &\leftrightarrow & (I-U^\dagger V)^{-1}U^\dagger Vx=\lambda x \leftrightarrow U^\dagger Vx= (I-U^\dagger V) \lambda x\nonumber\\
&\leftrightarrow& U^\dagger V(1+\lambda)x=\lambda x \leftrightarrow U^\dagger Vx=\displaystyle\frac{\lambda}{1+\lambda} x. \nonumber 
\end{eqnarray}
Assume that $T^\dagger V \in M^n_+(\mathbb{C})$. Then $\rho(T^\dagger V)\in \sigma(T^\dagger V)$ and $\mu=\displaystyle\frac{\lambda}{1+\lambda}$ achieves its maximum when $\lambda=\rho(T^\dagger V)$. Therefore $\rho(U^\dagger V)=\displaystyle\frac{\rho(T^\dagger V)}{1+\rho(T^\dagger V)}<1$. Conversely,  if $U^\dagger V\in M^n_+(\mathbb{C})$ and $\rho(U^\dagger V)=\displaystyle\frac{\rho(T^\dagger V)}{1+\rho(T^\dagger V)}<1$ then $\|U^\dagger V\|=\rho(U^\dagger V)<1$. Then $(I-U^\dagger V)^{-1}=\sum\limits_{k=0}^{\infty}(U^\dagger V)^k$. So that $T^\dagger V= (I-U^\dagger V)^{-1}U^\dagger V=\sum\limits_{k=1}^{\infty}(U^\dagger V)^k$. Thus, $T^\dagger V\in M^n_+(\mathbb{C})$. 

\end{proof}

\begin{rem}
Similar results to Proposition \ref{equiTdaggerV} are disseminated in the literature in different contexts. For example, in \cite[Theorem 2]{MR348984}, Berman and Plemmons proved that:

If $T=U-V$ is a proper splitting of $T\in M^{m\times n}(\mathbb{R})$ with $U^\dagger V\in \cI_K$ then, $T^\dagger V\in \cI_K$ if and only if $\rho(U^\dagger V)=\displaystyle\frac{\rho(T^\dagger V)}{1+\rho(T^\dagger V)}<1$. 

Evidently, the result holds if $\cI_K$ is replaced by $M_{\succ 0}^n(\mathbb{R})$. The particular case $\cI_K=M_{\succ 0}^n(\mathbb{R})$ together with other equivalent conditions can be found in \cite[Lemma 2.6]{MR1113154} and   \cite[Lemma 3.5]{MR3141710}. 

Later,   Climent and Perea proved that Berman and Plemmons' result holds if  $\cI_K$ is replaced by  $M_+^n(\mathbb{C})$ and $T,U$ are non singular matrices (see \cite[Theorem 2.5]{MR1695402}). Finally, in \cite[Theorem 3]{MR2198937}, Climent et al. unified \cite[Theorem 2.5]{MR1695402} and \cite[Lemma 2.6]{MR1113154} by considering a partial order induced by a positivity cone of matrices. 

Nevertheless, all these results include the general hypotheses $U^\dagger V\in \cI_K$ or $U^{-1} V\in M_+^n(\mathbb{C})$. Notice that in Proposition \ref{equiTdaggerV} we relaxed this general assumption. However, let us observe that this condition can not be omitted in the context of nonnegative matrices. In fact,  $T^{-1}V\succ 0$ does not imply $U^{-1}V\succ 0$, in general. For example, consider the proper splitting of  
$T=\left(\begin{array}{cc}
3 & -1 \\
0 & 2
\end{array}\right)$ given by
 $T=\left(\begin{array}{cc}
4 & 2 \\
1 & 2
\end{array}\right) - \left(\begin{array}{cc}
1 & 3 \\
1 & 0
\end{array}\right)$. Then,  $T^{-1}V=\left(\begin{array}{cc}
1/2 & 1 \\
1/2 & 0
\end{array}\right)\succ 0,$ but $U^{-1}V=\left(\begin{array}{cc}
0 & 1 \\
1/2 & -1/2
\end{array}\right)\nsucc 0$.

\end{rem}

\begin{pro}\label{H-rho-CS}
Let $T=U-V$ be proper splitting of $T$. Then the following assertions are equivalent:
\begin{enumerate}
\item $0\leq U^\dagger T\leq P_{T^*};$
\item $U^\dagger V\in M^n_+(\mathbb{C})$ and $\rho(U^\dagger V)=\displaystyle\frac{\rho(T^\dagger U)-1}{\rho(T^\dagger U)}<1$.
\end{enumerate}
\end{pro} 

\begin{proof}
Suppose that $0\leq U^\dagger T\leq P_{T^*}=U^\dagger U$. Then, $0\leq U^\dagger (U-T)=U^\dagger V$ and so $\rho(U^\dagger V)\in \sigma(U^\dagger V)$.  Then there exists $0\neq x\in\mathbb{C}^n$ such that $U^\dagger Vx=\rho(U^\dagger V)x$. So that $x\in\mathcal{R}(U^*)$ and then $x=U^\dagger Ux$. Therefore we get that $T^\dagger Ux=(I-U^\dagger V)^{-1}U^\dagger Ux= (I-U^\dagger V)^{-1}x=\displaystyle\frac{1}{1-\rho(U^\dagger V)}x$. Then $0\leq \displaystyle\frac{1}{1-\rho(U^\dagger V)}\leq \rho(T^\dagger U) $ and so $\rho(U^\dagger V)\leq \displaystyle\frac{\rho(T^\dagger U)-1}{\rho(T^\dagger U)}$. On the other hand, since $T^\dagger U\geq 0$ then $\rho(T^\dagger U)\in\sigma(T^\dagger U)$. Hence, there exists $0\neq y\in\mathbb{C}^n$ such that $T^\dagger Uy=\rho(T^\dagger U)y$. In consequence, $y\in\mathcal{R}(T^*)$ and so $y=U^\dagger Uy$.  Therefore, $\rho(T^\dagger U)y=T^\dagger Uy=(I-U^\dagger V)^{-1}U^\dagger Uy=(I-U^\dagger V)^{-1}y$. After some computations we get $U^\dagger V y=\displaystyle\frac{\rho(T^\dagger U)-1}{\rho(T^\dagger U)} y$.  Then $\displaystyle\frac{\rho(T^\dagger U)-1}{\rho(T^\dagger U)}\leq \rho(U^\dagger V)$. Therefore,  $\rho(U^\dagger V)=\displaystyle\frac{\rho(T^\dagger U)-1}{\rho(T^\dagger U)}<1$. Conversely, if item b) holds then $\|U^\dagger V\|=\rho(U^\dagger V)<1$. Then $(I-U^\dagger V)$ is invertible. So that $(I-U^\dagger V)^{-1}=\sum\limits_{k=0}^{\infty}(U^\dagger V)^k$. Then, by Proposition \ref{propiedades splitting}, we get  that $T^\dagger U=(I-U^\dagger V)^{-1}U^\dagger U=\sum\limits_{k=0}^{\infty}(U^\dagger V)^k U^\dagger U= P_{U^*}+\sum\limits_{k=1}^{\infty}(U^\dagger V)^k\geq0 $, or equivalently, $U^\dagger T=(T^\dagger U)^\dagger \geq 0.$ Moreover,  $0\leq U^\dagger V=U^\dagger (U-T)=P_{T^*}-U^\dagger T.$ The proof is complete.
\end{proof}

\begin{rem} We recommend \cite{MR1113154}, \cite{MR3141710} and \cite{MR1695402} for similar results to Proposition \ref{H-rho-CS} but for the partial order $\succ.$ On the other hand, a similar result to the previous one but considering $T\in M^{n}(\mathbb{C})$ non singular  and under the general hypothesis that $U^{-1}V> 0$,  can be found in \cite[Theorem 2.5]{MR1695402} .

Let us observe that in the context of nonnegative matrices if $T=U-V$ is a proper splitting of $T$ then  $U^\dagger T\succ 0$ does not imply $U^{\dagger}V\succ 0$, in general. For example, consider the proper splitting of  
$T=\left(\begin{array}{cc}
1 & 0 \\
1 & 2
\end{array}\right)$ given by
 $T=\left(\begin{array}{cc}
1 & 0 \\
0 & 2
\end{array}\right) - \left(\begin{array}{cc}
0 & 0 \\
-1 & 0
\end{array}\right)$. Then,  $U^\dagger T=\left(\begin{array}{cc}
1 & 0 \\
1/2 & 1
\end{array}\right)\succ 0$ but $U^{\dagger}V=\left(\begin{array}{cc}
0 & 0 \\
-1/2 & 0
\end{array}\right)\nsucc 0$.

\end{rem}

\begin{pro}\label{convergencia con solucion positiva}
Let $T=U-V$ be a proper splitting of $T$. If there exists $\tilde{X}\in M^n_H(\mathbb{C})$ such that $U\tilde{X}=V$ and $\rho(\tilde{X})<1$ then $\rho(U^\dagger V)\leq\rho(\tilde{X})<1$.
\end{pro}

\begin{proof}

Let $\tilde{X}\in M^n_H(\mathbb{C})$ such that $U\tilde{X}=V$ and $\rho(\tilde{X})<1.$ Then, by Theorem \ref{soluciones} and Proposition \ref{UdaggerV-Hermitiana}, it holds that $U^\dagger V\in M^n_H(\mathbb{C})$ and so $\rho(U^\dagger V)=\|U^\dagger V\|$. Now, since  $\tilde{X}=U^\dagger V+Z$ for some $Z\in M^n(\mathbb{C})$ with $\mathcal{R}(Z)\subseteq \mathcal{N}(U)$, we get that $\rho(U^\dagger V)=\|U^\dagger V\|\leq \|\tilde{X}\|=\rho(\tilde{X})<1$. 

\end{proof}

Once the converge of the iterative process is guaranteed, then the second problem to be addressed is to improve the speed of convergence of the method. To this end, we provide two comparison results.

\begin{thm}\label{comp1}
Let $T=U_1-V_1=U_2-V_2$ be two proper splittings of $T$ such that $0\leq T^\dagger V_1\leq T^\dagger V_2$ and $\mathcal{N}(V_1)=\mathcal{N}(V_2)$.  Then $\rho(U_1^\dagger V_1)\leq \rho(U_2^\dagger V_2)<1$.
\end{thm}

\begin{proof}
Since $0\leq T^\dagger V_1\leq T^\dagger V_2$ and $ \mathcal{N}(V_1)=\mathcal{N}(V_2)$ we get that $\mathcal{N}(T^\dagger V_1)= \mathcal{N}(V_1)=\mathcal{N}(V_2)=\mathcal{N}(T^\dagger V_2)$ and $\mathcal{R}(T^\dagger V_1)=\mathcal{R}(T^\dagger V_2)$. Now, observe that $\mathcal{R}(T^\dagger V_2 - T^\dagger V_1)\cap\mathcal{N}(T^\dagger V_2)\subseteq \mathcal{R}(T^\dagger V_2)\cap\mathcal{N}(T^\dagger V_2)=\{0\}$. On the other hand, $\mathcal{R}((T^\dagger V_1)^\dagger - (T^\dagger V_2)^\dagger)\cap\mathcal{N}(T^\dagger V_1)\subseteq \mathcal{N}( V_1)^\bot\cap \mathcal{N}(T^\dagger V_1)= \mathcal{N}(V_1)^{\bot}\cap \mathcal{N}(V_1)=\{0\}$. So that, by Theorem \ref{lowner-mp} it holds that $0\leq (T^\dagger V_2)^\dagger \leq (T^\dagger V_1)^\dagger$. Now, since $T^\dagger=(I-U_1^\dagger V_1)^{-1}U_1^\dagger=(I-U_2^\dagger V_2)^{-1}U_2^\dagger$ it holds that 

\begin{eqnarray}
0\leq \left( T^\dagger V_2 \right)^\dagger \leq \left(T^\dagger V_1\right)^\dagger &\leftrightarrow& \left(  (I-U_2^\dagger V_2)^{-1}U_2^\dagger V_2 \right)^\dagger \leq \left(   (I-U_1^\dagger V_1)^{-1}U_1^\dagger V_1 \right)^\dagger \nonumber\\
&\leftrightarrow& 0\leq \left(U_2^\dagger V_2\right)^\dagger \left( I - U_2^\dagger V_2\right)\leq (U_1^\dagger V_1)^\dagger (I-U_1^\dagger V_1) \nonumber\\
&\leftrightarrow& 0\leq   \left(U_2^\dagger V_2\right)^\dagger - P_{\mathcal{N}(U_2^\dagger V_2)^\bot} \leq  \left(U_1^\dagger V_1\right)^\dagger - P_{\mathcal{N}\left(U_1^\dagger V_1\right)^\bot}\nonumber\\
&\leftrightarrow& 0\leq \left(U_2^\dagger V_2\right)^\dagger - P_{V_1^*} \leq  \left(U_1^\dagger V_1\right)^\dagger  - P_{V_1^*}\nonumber\\
&\leftrightarrow& 0\leq \left(U_2^\dagger V_2\right)^\dagger \leq \left(U_1^\dagger V_1\right)^\dagger\nonumber\\
& \leftrightarrow& 0\leq  U_1^\dagger V_1\leq U_2^\dagger V_2. \nonumber
\end{eqnarray}
Then, by Proposition \ref{equiTdaggerV}, we get that $\rho(U_1^\dagger V_1)\leq \rho(U_2^\dagger V_2)<1$.
 \end{proof}



\begin{thm} \label{comp2}
Let $T=U_1-V_1=U_2-V_2$ be two proper splittings of $T$. If $0\leq T^\dagger U_1\leq T^\dagger U_2 \leq P_{T^*}$ then $\rho(U_1^\dagger V_1)\leq \rho(U_2^\dagger V_2)<1$.
\end{thm}

\begin{proof}
By Proposition \ref{H-rho-CS}, it holds that $\rho(U_1^\dagger V_1)<1$ and $\rho(U_2^\dagger V_2)<1$.  Now, observe that
\begin{eqnarray}
0\leq T^\dagger U_1 \leq T^\dagger U_2 \leq P_{T^*} &\leftrightarrow& 0\leq  (I-U_1^\dagger V_1)^{-1} U_1^\dagger U_1 \leq (I-U_2^\dagger V_2)^{-1} U_2^\dagger U_2 \leq P_{T^*} \nonumber\\
&\leftrightarrow& 0\leq  (I-U_1^\dagger V_1)^{-1} P_{T^*} \leq (I-U_2^\dagger V_2)^{-1} P_{T^*} \leq P_{T^*}\nonumber\\
&\leftrightarrow& 0\leq \left(  (I-U_2^\dagger V_2)^{-1} P_{T^*}  \right)^\dagger \leq \left(  (I-U_1^\dagger V_1)^{-1} P_{T^*}  \right)^\dagger \leq P_{T^*}\nonumber\\
&\leftrightarrow& 0\leq P_{T^*} (I-U_2^\dagger V_2) \leq P_{T^*} (I-U_1^\dagger V_1)\leq P_{T^*}\nonumber\\
&\leftrightarrow& 0\leq P_{T^*} -U_2^\dagger V_2\leq P_{T^*} -U_1^\dagger V_1\leq P_{T^*}\nonumber\\ 
&\leftrightarrow&  0\leq U_1^\dagger V_1\leq U_2^\dagger V_2\leq P_{T^*} ,\nonumber
\end{eqnarray}
where the third equivalence holds because  $r((I-U_1^\dagger V_1)^{-1} P_{T^*} ) =r(P_{T^*})$ and $r((I-U_2^\dagger V_2)^{-1} P_{T^*} ) =r(P_{T^*})$ (see Theorem \ref{lowner-mp}). Therefore,  $\rho(U_1^\dagger V_1)\leq \rho(U_2^\dagger V_2)<1$.
\end{proof}

\begin{rem}
Comparison results similar to those above but for the partial order $\succ,$ can be found in \cite{MR1113154}, \cite{MR3141710}. When contrasting our results with those for the partial order $\succ$, it can be observed that we require extra conditions on the matrices of the splittings in order to get the conclusions. More precisely, we require that $\mathcal{N}(V_1)=\mathcal{N}(V_2)$ in Theorem \ref{comp1} and that $T^\dagger U_1\leq T^\dagger U_2 \leq P_{T^*}$ in Theorem \ref{comp2}.  We leave as an open problem to determine if these conditions can be relaxed. 

\end{rem}

\section{Some examples of proper splittings}\label{Sec-Ejemplos}

Recall that every $T\in M^{m\times n}(\mathbb{C})$ can be factorized as $T=U|T|$, where $U\in M^{m\times n}(\mathbb{C})$ is a partial isometry and $|T|=(T^*T)^{1/2}$. Such a factorization is called a polar decomposition of $T$ and the matrix $U$ is unique if $\mathcal{R}(U)=\mathcal{R}(T)$. In what follows, we call {\bf{the polar decomposition}} of $T$ to the unique factorization $T=U_T|T|$  where $U_T\in M^{m\times n}(\mathbb{C})$ is a partial isometry with $\mathcal{R}(U_T)=\mathcal{R}(T)$. Notice that it also holds that $\mathcal{N}(U_T)=\mathcal{N}(T)$. Therefore, the partial isometry $U_T$ can be used to define a proper splitting of $T$, $T=U_T-V$. This proper splitting has the advantage that $U_T^\dagger=U_T^*$. The aim in this section is to study this particular proper splitting. We study equivalent conditions to guarantee the convergence of this particular proper splitting and we compare it with some proper splittings defined for certain classes of matrices.

\begin{pro}\label{splittingpolar}
Let  $T\in M^{m\times n}(\mathbb{C})$. Then $T=U_T - V$ is a proper splitting of $T$ such that $U_T^*V \in M^n_H(\mathbb{C})$.
\end{pro}

\begin{proof}
Since $\mathcal{R}(U_T)=\mathcal{R}(T)$ and $\mathcal{N}(U_T)=\mathcal{N}(T)$ then $T=U_T-V$ is a proper splitting. Now, as $U_T^*V=U_T^*(U_T-T)=P_{T^*}-|T|$ then $U_T^*V$ is Hermitian. 
\end{proof}

%
%
%
%
%

The next result due to Baksalary, Liski and Trenkler \cite{MR1032688} will be used to prove Proposition \ref{polar splitting}. We present an alternative proof.
\begin{lem}\label{BLT}
Let $A,B\in M_+^n(\mathbb{C})$. Then, $A\leq B$ if and only if $\rho({B^\dagger A})\leq 1$ and $\mathcal{R}(A)\subseteq \mathcal{R}(B)$.
\end{lem}

\begin{proof}
Let $A,B\in M_+^n(\mathbb{C})$. If $A\leq B$ then, by Douglas' theorem \cite[Theorem 1]{MR0203464}, $\mathcal{R}(A)\subseteq \mathcal{R}(B)$ and  $\|(B^{1/2})^\dagger A^{1/2}\|=\inf\{\lambda: A\leq \lambda B\}$. Thus, $\|(B^{1/2})^\dagger A^{1/2}\|\leq 1$. Therefore, $\rho(B^\dagger A)=\rho(A^{1/2}B^\dagger A^{1/2})=\| A^{1/2}B^\dagger A^{1/2}\|=\|(B^{1/2})^\dagger A^{1/2}\|^2\leq 1$.  

Conversely, suppose that  $\rho({B^\dagger A})\leq 1$ and $\mathcal{R}(A)\subseteq \mathcal{R}(B)$. Hence, since $\sigma(A^{1/2}B^\dagger A^{1/2})=\sigma(B^\dagger A)$ we get that $\|(B^{1/2})^\dagger A^{1/2}\|^2 = \|A^{1/2}B^\dagger A^{1/2}\|=\rho(A^{1/2}B^\dagger A^{1/2})=\rho(B^\dagger A)\leq 1$. Then, applying again \cite[Theorem 1]{MR0203464},  $\inf\{\lambda: A\leq \lambda B\}=\|(B^{1/2})^\dagger A^{1/2}\|\leq 1$. Therefore $A\leq B$.

\end{proof}

In the next proposition we provide sufficient conditions for the convergence of the proper splitting $T=U_T-V$.
\begin{pro} \label{polar splitting} Let  $T\in M^{m\times n}(\mathbb{C})$ and consider the proper splitting $T=U_T - V$ of $T$. The following conditions are equivalent:
\begin{enumerate}
\item  $T^\dagger V \in M^n_+(\mathbb{C})$;
\item $U_T^*V \in M^n_+(\mathbb{C})$;
\item $\|T\|\leq 1$.
\end{enumerate}
 If any of the above conditions holds then the proper splitting $T=U_T-V$ is convergent.
\end{pro}

\begin{proof}
$a)\leftrightarrow b)$ Suppose that $T^\dagger V \in M^n_+(\mathbb{C})$. Since $T^\dagger V=|T|^\dagger- P_{T^*}\geq 0$ then, by Theorem \ref{lowner-mp}, it holds $|T|\leq P_{T^*}$. Therefore, $U_T^*V=U_T^*(U_T-T)=P_{T^*}-|T|\in M^n_+(\mathbb{C})$. The converse is similar. 

$b)\leftrightarrow c)$ If $U_T^*V=P_{T^*}-|T|\geq 0$ then $|T|\leq P _{T^*}$ and so $\|T\|=\| |T|\|\leq 1$. Conversely, if $\|T\|\leq 1$ we get that $\|P_{T^*}|T|\|=\||T|\|=\|T\|\leq 1$. Thus, $\rho(P_{T^*}|T|)\leq 1$. So that, as $\mathcal{R}(|T|)=\mathcal{R}(P_{T^*})$,  by Lemma \ref{BLT} we get that $|T|\leq P_{T^*}$. Therefore, $U_T^*V\in  M^n_+(\mathbb{C})$. 

The last part follows from Proposition \ref{equiTdaggerV}.

\end{proof}

The remain of this section is devoted to study proper splittings for normal matrices and for matrices that can be factorized as the product of two orthogonal projections. 

To this end, we denote by $\mathcal{P}\cdot \mathcal{P}=\{T\in M^{n}(\mathbb{C}): T=P_1P_2 \; \text{with} \; P_1,P_2\in\cP\}$, where $\cP=\{Q\in M^{n}(\mathbb{C}): Q=Q^2=Q^*\}$. This set was studied by Corach and Maestripieri in \cite{MR2775769} in the more general context of bounded linear operators defined on a Hilbert space. In the finite dimensional case, they proved that if $T\in \mathcal{P}\cdot \mathcal{P}$ then  $T=P_TP_{T^*}$ and $\mathcal{R}(T)\dot+\mathcal{N}(T)=\mathbb{C}^n$ (see \cite[Theorem 3.2]{MR2775769}). In such case, notice that the projection $Q:=Q_{\mathcal{R}(T)//\mathcal{N}(T)}$ is well-defined  and it can be used to define a proper splitting of $T$, $T=Q-V_1.$ In the next proposition we study the convergence of this proper splitting and we compare it with  $T=U_T-V$.

\begin{pro}\label{PP}
Let $T\in \mathcal{P}\cdot\mathcal{P}$ and  $Q=Q_{\mathcal{R}(T)//\mathcal{N}(T)}$. Then the following assertions hold:
\begin{enumerate}
\item The proper splitting $T=Q-V_1$ is convergent.
\item  The proper splitting $T=U_T-V$ is convergent.
\item $\rho(U_T^*V)\leq \rho(Q^\dagger V_1)<1$.
\end{enumerate}
\end{pro}

\begin{proof} 
$a)$ By Proposition \ref{equiTdaggerV}, it suffices to prove that  $T^\dagger V_1\in M^n_+(\mathbb{C})$. Now, by  \cite[Theorem 4.1]{MR2653816}, it holds that $T^\dagger =Q^*$. Then $T^\dagger V_1=T^\dagger(Q-T)=Q^*Q-P_{T^*}\geq 0$. Thus, $T^\dagger V_1\in M^n_+(\mathbb{C})$ and the splitting $T=Q-V_1$ is convergent.

$b)$ Since $T\in\mathcal{P}\cdot\mathcal{P}$ then $T^*\in\mathcal{P}\cdot\mathcal{P}$. So that $0\leq T^*T=P_{T^*} P_{T}P_{T^*}=P_{T^*}T\leq P_{T^*}$.  Therefore, $|T|\leq P_{T^*}$ and so, $U_T^*V=P_{T^*}-|T|\geq 0$. Therefore, by Proposition  \ref{polar splitting}  we get that the proper splitting $T=U_T-V$ is convergent.

$c)$ By the above items,  the proper splittings $T=Q-V_1$ and $T=U_T-V$ are convergent. As $Q^\dagger = P_{T^*}P_T$  (see \cite[Theorem 4.1]{MR2653816}) we get that
\begin{eqnarray}
Q^\dagger V_1 &=& P_{T^*}P_TV_1=P_{T^*}P_T(Q-T)\nonumber\\
&=& P_{T^*}(Q-T)= P_{T^*}-P_{T^*}P_TP_{T^*},\nonumber
\end{eqnarray}
and 
\begin{eqnarray}
U_T^\dagger V &=& U_T^\dagger (U_T-T)=P_{T^*}-|T|\nonumber\\
&=& P_{T^*}-|P_TP_{T^*}|.\nonumber
\end{eqnarray}
Since $0\leq P_{T^*}P_TP_{T^*}\leq I$ then $0\leq P_{T^*}P_TP_{T^*}\leq (P_{T^*}P_TP_{T^*})^{1/2}\leq I$ and so $P_{T^*}P_TP_{T^*}\leq |P_TP_{T^*}|$. Hence, $0\leq U_T^\dagger V\leq Q^\dagger V_1$, where $0\leq U_T^\dagger V$ holds by the proof of item b). Therefore,  $\rho(U_T^*V)\leq \rho(Q^\dagger V_1)<1$.
\end{proof}

 \begin{exa}
 Observe that the inequality $\rho(U_T^*V)\leq \rho(Q^\dagger V_1)$ can be strict. Take 
 $T=\left(\begin{array}{cc}
 1/2 & 0\\
 -1/2 & 0
 \end{array}\right)= \left(\begin{array}{cc}
 1/2 & -1/2\\
 -1/2 & 1/2
 \end{array}\right) \left(\begin{array}{cc}
 1 & 0\\
 0 & 0
 \end{array}\right)\in\mathcal{P} \cdot \mathcal{P}$.
 \end{exa}
Consider the proper splitting $T=U_T-V= \left(\begin{array}{cc}
 \sqrt{2}/2 & 0\\
 -\sqrt{2}/2 & 0
 \end{array}\right).$ Now,  since $U_T^*V=\left(\begin{array}{cc}
 \frac{2-\sqrt{2}}{2} & 0\\
 0 & 0
 \end{array}\right)$ then $\rho(U_T^*V)= \frac{2-\sqrt{2}}{2}$.  On the other hand, consider the proper splitting $T=Q-V_1=\left(\begin{array}{cc}
 1& 0\\
 -1 & 0
 \end{array}\right)-V_1$. Then, $Q^\dagger V_1=\left(\begin{array}{cc}
 1/2 & 0\\
 0 & 0
 \end{array}\right)$ and so $\rho(Q^\dagger V_1) = \frac{1}{2}$. Therefore $\rho(U_T^*V)<\rho(Q^\dagger V_1)<1$.

\

We end this section by studying proper splittings of normal matrices. More precisely, given a normal matrix $T$ we study the proper splitting $T=P_T-V_2$ and we compare it with $T=U_T-V$. 

\begin{pro}\label{normales} Let  $T\in M^{n}(\mathbb{C})$ be a normal matrix. Then, the following assertions hold:
\begin{enumerate}
\item The proper splitting $T=P_T-V_2$ is convergent if and only if $\|P_T-T\|<1$.
\item If $T\in M_H^{n}(\mathbb{C})$, $P_T-|T|\in M_+^n(\mathbb{C})$ and  $\|P_T-T\|<1$ then the proper splittings  $T=P_T-V_2$ and  $T=U_T-V$ are convergent and $\rho(U_T^*V)\leq\rho(V_2)<1$.
\end{enumerate}  
\end{pro}

\begin{proof}
 $a)$ Observe that if $T$ is normal then  $V_2=P_T-T$ is also normal. Therefore, $\rho(P_TV_2)=\rho(V)=\|P_T-T\|$. Thus,  $T=P_T-V$ is convergent if and only if $\|P_T-T\|<1$.

 $b)$ Let $T\in M_H^n(\mathbb{C})$ and consider the proper splittings $T=P_T-V=U_T-V_1$. Since $\|P_T-T\|<1$ and $U_T^*V=P_T-|T|\in M_+^n(\mathbb{C})$ then both splittings are convergent.  Moreover, as $T\in M_H^n(\mathbb{C})$ then $T\leq |T|$. Thus,  $0\leq U_T^*V=U_T^*(U_T-T)=P_T-|T|\leq P_T-T=V_2=P_TV_2$. Therefore, $\rho(U_T^*V)=\|P_T-|T|\|\leq \|P_T-T\|=\rho(V_2)<1$. 
\end{proof}

Given a normal matrix $T\in M^{n}(\mathbb{C})$, then a natural proper splitting to be considered is: $T=|T|-V_3$.  However, it can be proved that if $\rho(|T|^\dagger V_3)<1$ then $T\in M^n_+(\mathbb{C})$. Thus, $T=|T|$ and the proposed splitting has not any sense.

%
%
%

\section{Conclusions}\label{conclusion}
Proper splitting of matrices was introduced by Berman and Plemmons with the aim of developing an iterative process to get the best least square approximate solution of $Tx = w$.
In this article we extended the usefulness of proper splittings by proposing an iterative process to approximate a reduced solution of the matrix equation $TX=W.$ We observed that given a proper splitting $T=U-V\in M^{m\times n}(\mathbb{C})$, the proposed iterative process converges if and only if $\rho(U^\dagger V)<1$. Moreover, we provided equivalent conditions to $\rho(U^\dagger V)<1$ and some speed comparison results by using the L\"owner order of matrices. Furthermore, we presented and studied some examples of proper splittings. We believe that these results and examples extend the knowledge about splitting of matrices and its applications. In addition,  we think that this new approach leaves room for additional research on this topic.    

%

\begin{thebibliography}{10}

\bibitem{MR2407082}
M.~L. Arias, G. Corach, and M.~C. Gonzalez, \emph{Generalized
  inverses and {D}ouglas equations}, Proc. Amer. Math. Soc. \textbf{136}
  (2008), no.~9, 3177--3183. 


\bibitem{MR2514477}
M.~L. Arias and M.~C. Gonzalez, \emph{Reduced solutions of {D}ouglas
  equations and angles between subspaces}, J. Math. Anal. Appl. \textbf{355}
  (2009), no.~1, 426--433. 



\bibitem{MR1032688}
J.~K. Baksalary, E.~P. Liski, and G. Trenkler, \emph{Mean square
  error matrix improvements and admissibility of linear estimators}, J.
  Statist. Plann. Inference \textbf{23} (1989), no.~3, 313--325. 
\bibitem{MR1048801}
J.~K. Baksalary, K. Nordstr\"{o}m, and G.~P.~H. Styan,
  \emph{L\"{o}wner-ordering antitonicity of generalized inverses of {H}ermitian
  matrices}, Linear Algebra Appl. \textbf{127} (1990), 171--182. 
	
\bibitem{286957}
R.~T. {Behrens} and L.~L. {Scharf}, \emph{Signal processing applications of
  oblique projection operators}, IEEE Transactions on Signal Processing
  \textbf{42} (1994), no.~6, 1413--1424.

\bibitem{MR348984}
A. Berman and R.~J. Plemmons, \emph{Cones and iterative methods for
  best least squares solutions of linear systems}, SIAM J. Numer. Anal.
  \textbf{11} (1974), 145--154. 


\bibitem{MR2512471}
R. Boyer and G. Bouleux, \emph{Oblique projections for
  direction-of-arrival estimation with prior knowledge}, IEEE Trans. Signal
  Process. \textbf{56} (2008), no.~4, 1374--1387. 

\bibitem{BOYER20092547}
R. Boyer, \emph{Oblique projection for source estimation in a competitive
  environment: Algorithm and statistical analysis}, Signal Processing
  \textbf{89} (2009), no.~12, 2547 -- 2554, Special Section: Visual Information
  Analysis for Security.


\bibitem{MR2115593}
Y.-L. Chen and X.-Y. Tan, \emph{Computing generalized inverses of
  matrices by iterative methods based on splittings of matrices}, Appl. Math.
  Comput. \textbf{163} (2005), no.~1, 309--325. 
	
\bibitem{MR2198937}
J.-J. Climent, V. Herranz, and C. Perea, \emph{Positive cones
  and convergence conditions for iterative methods based on splittings}, Linear
  Algebra Appl. \textbf{413} (2006), no.~2-3, 319--326. 
	
\bibitem{MR1628383}
J.-J. Climent and C. Perea, \emph{Some comparison theorems for weak
  nonnegative splittings of bounded operators}, Proceedings of the {S}ixth
  {C}onference of the {I}nternational {L}inear {A}lgebra {S}ociety ({C}hemnitz,
  1996) \textbf{275/276} (1998), 77--106. 

\bibitem{MR1695402}
\bysame, \emph{Convergence and comparison theorems for multisplittings},
  \textbf{6} (1999), Czech-US Workshop in Iterative Methods and Parallel Computing,
  Part 2 (Milovy, 1997), 93--107. 

\bibitem{MR2653816}
G.~Corach and A.~Maestripieri, \emph{Polar decomposition of oblique
  projections}, Linear Algebra Appl. \textbf{433} (2010), no.~3, 511--519.
  

\bibitem{MR2775769}
\bysame, \emph{Products of orthogonal projections and polar decompositions},
  Linear Algebra Appl. \textbf{434} (2011), no.~6, 1594--1609.

\bibitem{MR745083}
G. Csordas and R.~S. Varga, \emph{Comparisons of regular splittings of
  matrices}, Numer. Math. \textbf{44} (1984), no.~1, 23--35. 

\bibitem{MR0203464}
R.~G. Douglas, \emph{On majorization, factorization, and range inclusion of
  operators on {H}ilbert space}, Proc. Amer. Math. Soc. \textbf{17} (1966),
  413--415. 

\bibitem{MR1969059}
L. Elsner, A. Frommer, R. Nabben, H. Schneider, and D.~B.
  Szyld, \emph{Conditions for strict inequality in comparisons of spectral
  radii of splittings of different matrices}, \textbf{363}, 2003, Special issue on
  nonnegative matrices, $M$-matrices and their generalizations (Oberwolfach,
  2000), ~65--80. 



\bibitem{MR210720}
T.~N.~E. Greville, \emph{Note on the generalized inverse of a matrix product},
  SIAM Rev. \textbf{8} (1966), 518--521; erratum, ibid. 9 (1966), 249.
 

\bibitem{MR0075677}
A. ~S. Householder, \emph{On the convergence of matrix iterations}, Rep.
  ORNL 1883, Oak Ridge National Laboratory, Oak Ridge, Tenn., 1955.
  

\bibitem{John}
F.~John., \emph{Advanced numerical analysis}, Lecture Notes, Dept. of
  Mathematics, New York Univ, 1956.


\bibitem{MR3671533}
M.~R.~Kannan, \emph{P-proper splittings}, Aequationes Math. \textbf{91}
  (2017), no.~4, 619--633. 

\bibitem{MR417212}
C.~G. Khatri and S.~K. Mitra, \emph{Hermitian and nonnegative definite
  solutions of linear matrix equations}, SIAM J. Appl. Math. \textbf{31}
  (1976), no.~4, 579--585. 

\bibitem{KHORSANDZAK2013269}
M. {Khorsand Zak} and F. Toutounian, \emph{Nested splitting conjugate
  gradient method for matrix equation AXB=C and preconditioning}, Comput. Math. Appl. \textbf{66} (2013), no.~3, 269 -- 278.

\bibitem{MR266409}
O.~L. Mangasarian, \emph{A convergent splitting of matrices}, Numer. Math.
  \textbf{15} (1970), 351--353. 
	
\bibitem{MATHEW1997129}
T. Mathew and K. Nordstr\"om, \emph{Wishart and chi-square
  distributions associated with matrix quadratic forms}, J. Multivar. Anal \textbf{61} (1997), no.~1, 129 -- 143.

\bibitem{MR3141710}
D. Mishra, \emph{Nonnegative splittings for rectangular matrices},
  Comput. Math. Appl. \textbf{67} (2014), no.~1, 136--144. 
	
\bibitem{MR1368073}
R. Nabben, \emph{A note on comparison theorems for splittings and
  multisplittings of {H}ermitian positive definite matrices}, Linear Algebra
  Appl. \textbf{233} (1996), 67--80. 

\bibitem{MR2490784}
B. Na\v{c}evska, \emph{Iterative methods for computing generalized
  inverses and splittings of operators}, Appl. Math. Comput. \textbf{208}
  (2009), no.~1, 186--188. 
	
\bibitem{MR215487}
J.~M. Ortega and W.~C. Rheinboldt, \emph{Monotone iterations for
  nonlinear equations with application to {G}auss-{S}eidel methods}, SIAM J.
  Numer. Anal. \textbf{4} (1967), 171--190. 
	


\bibitem{RAOCHAGANTY1997421}
N.~{R. Chaganty} and A.K. Vaish, \emph{An invariance property of common
  statistical tests},  Linear Algebra Appl. \textbf{264} (1997),
  421 -- 437, Sixth Special Issue on Linear Algebra and Statistics.

\bibitem{MR1113154}
Y. Song, \emph{Comparisons of nonnegative splittings of matrices},
  Linear Algebra Appl. \textbf{154/156} (1991), 433--455. 

\bibitem{MR1930390}
Y. Song, \emph{Comparison theorems for splittings of matrices}, Numer.
  Math. \textbf{92} (2002), no.~3, 563--591.

\bibitem{MR309971}
J.~S. Vandergraft, \emph{Applications of partial orderings to the study of
  positive definiteness, monotonicity, and convergence of iterative methods for
  linear systems}, SIAM J. Numer. Anal. \textbf{9} (1972), 97--104. 
	
\bibitem{MR0158502}
R.~S. Varga, \emph{Matrix iterative analysis}, Prentice-Hall, Inc.,
  Englewood Cliffs, N.J., 1962. 

\bibitem{MR1286436}
Z.~I. Wo\'{z}nicki, \emph{Nonnegative splitting theory}, Japan J. Indust.
  Appl. Math. \textbf{11} (1994), no.~2, 289--342. 
\end{thebibliography}
%

\end{document}